\documentclass[a4paper,11pt]{amsproc}
\usepackage{amsmath}
\usepackage{amssymb}
\usepackage{amsthm}
\usepackage{mathrsfs}
\usepackage{soul}
\usepackage{tikz} 
\usepackage{float}

	\usepackage{pgf}
\usepackage{tikz}
\usetikzlibrary{arrows,automata}
\usepackage[latin1]{inputenc}

\theoremstyle{plain}
\usepackage{color}
\usepackage[colorlinks]{hyperref}

\newtheorem{theorem}{Theorem}[section]

\newtheorem{corollary}[theorem]{Corollary}
\theoremstyle{definition}
\newtheorem{definition}[theorem]{Definition}

\newtheorem{notation}[theorem]{Notation}
\newtheorem{question}[theorem]{Question}
\newtheorem{example}[theorem]{Example}

\theoremstyle{remark}
\newtheorem{remark}[theorem]{Remark}
\numberwithin{equation}{section}

\newcommand{\R}{\mathbb{R}}
\newcommand{\Q}{\mathbb{Q}}

\newcommand{\ccx}{C_c(X)}
\newcommand{\acx}{A_c(X)}
\newcommand{\gnccx}{\Gamma(\ccx) }
\newcommand{\gccx}{\Gamma^{'}_2(\ccx) }

\newcommand{\pccx}{ \mathcal{P}(\ccx) }
\newcommand{\pacx}{ \mathcal{P}(\acx) }
\newcommand{\gnacx}{\Gamma(\acx) }
\newcommand{\vn}{V_0(C_c(X))}

\newcommand{\vtccx}{V_2(\ccx)}

\newcommand{\cx}{C(X)}
\newcommand{\w}{\widehat}
\newcommand{\gncx}{\Gamma(\cx) }
\newcommand{\gcx}{\Gamma^{'}_2(\cx) }

\begin{document}	
	\title[Zero-divisor graph and comaximal graph of $ \ccx $]{Zero-divisor graph and comaximal graph of rings of continuous functions with countable range	}
		\author{Rakesh Bharati}
		\address[Rakesh Bharati]{Department of Pure Mathematics, University of Calcutta, 35, Ballygunge Circular Road, Kolkata - 700019, INDIA} 
	\email{rbharati.rs@gmail.com}
		
	\author{Amrita Acharyya}	\address[Amrita Acharyya]{Department of Mathematics and Statistics, University of Toledo, Main Campus,
		Toledo, OH 43606-3390}\email {amrita.acharyya@utoledo.edu}

		\author{A. Deb Ray }
		\address[A. Deb Ray]{Department of Pure Mathematics, University of Calcutta, 35, Ballygunge Circular Road, Kolkata - 700019, INDIA}
		\email{debrayatasi@gmail.com}
	
		\author{ Sudip Kumar Acharyya}
		\address[Sudip Kumar Acharyya]{Department of Pure Mathematics, University of Calcutta, 35, Ballygunge Circular Road, Kolkata - 700019, INDIA} 
		\email{sdpacharyya@gmail.com}

		\thanks{The first author acknowledges financial support from University Grants Commission, New Delhi, for the award of research fellowship (F.No. 16-9(June 2018)/2019 (NET/CSIR))}

		\begin{abstract}
		
		 In this paper, two outwardly different graphs, namely, the zero divisor graph $\Gamma(C_c(X))$ and the comaximal graph $\Gamma_2^{'}(C_c(X))$ of the ring $C_c(X)$ of all real-valued continuous functions having countable range, defined on any Hausdorff zero dimensional space $X$, are investigated. It is observed that these two graphs exhibit resemblance, so far as the diameters, girths, connectedness, triangulatedness or hypertriangulatedness. are concerned. However, the study reveals that the zero divisor graph $\Gamma(A_c(X))$ of an intermediate ring $A_c(X)$ of $C_c(X)$ is complemented if and only if the space of all minimal prime ideals of $A_c(X)$ is compact. Moreover, $\Gamma(C_c(X))$ is complemented when and only when its subgraph $\Gamma(A_c(X))$ is complemented. On the other hand,  the comaximal graph of $C_c(X)$ is complemented if and only if the comaximal graph of its over-ring $C(X)$ is complemented and the latter graph is known to be complemented if and only if $X$ is a $P$-space. Indeed, for a large class of spaces (i.e., for perfectly normal, strongly zero dimensional spaces which are not P-spaces), $\Gamma(C_c(X))$ and $\Gamma_2^{'}(C_c(X))$ are seen to be non-isomorphic. Defining appropriately the quotient of a graph, it is utilised to establish that for a discrete space $X$, $\Gamma(C_c(X))$ (= $\Gamma(C(X))$) and $\Gamma_2^{'}(C_c(X))$  (= $\Gamma_2^{'}(C(X))$) are isomorphic, if $X$ is atmost countable. Under the assumption of continuum hypothesis, the converse of this result is also shown to be true.

		\end{abstract}
		
	\maketitle
	
	\section{Introduction}
	We start with a Hausdorff zero-dimensional topological space $ X $. Let $ \ccx $ denote the set of all real-valued continuous functions on $ X $ which have countable range. It is well known that $ \ccx $ is a subring as well as a sublattice of the familiar ring $ C(X) $ of all real-valued continuous functions on $ X $. People have started investigating the algebraic and lattice properties of the rings $ \ccx $ and $ C_c^*(X)=\ccx\cap C^*(X) $ vis-a-vis corresponding topological properties of $ X $ only recently. We refer the articles \cite{ref1}, \cite{ref2}, \cite{ref3}, \cite{ref10} in this context. It is rightly remarked in paragraph 5 of the introductory section of the article \cite{ref10} that $ \ccx $ proves to be a good companion for $ C(X) $ in pegging some topological properties of $ X $ to appropriate algebraic properties of $ C(X) $ or to that of $ \ccx $. Our intention in this article is to study some relevant properties of each of the two graphs viz. the zero-divisor graph $ \gnccx $ and the comaximal graph $ \gccx $ of the ring $ \ccx $. We would like to mention that in the literature there is only one paper on the zero-divisor graph $ \Gamma(C(X)) $ of the ring $ C(X) $ \cite{ref4} and also a solitary article concerning the comaximal graph $ \Gamma_2^{'}(C(X)) $ of $ C(X) $ \cite{ref6}. It is realized that there is an interplay between the graph properties of $ \gnccx $ (respectively $ \gccx $) and the ring properties of $ \ccx $, leading to further interaction between these two properties and the topological properties of $ X $. See Theorem \ref{t-2.8}, Theorem \ref{t-3.13}, Theorem \ref{t-4.5} and Theorem \ref{t-4.12} in this connection. These may be called the countable counterparts of the corresponding theorems in \cite{ref4} and in \cite{ref6}. However in order to establish a number of properties in the present paper, we take to our advantage the presence of characteristic functions of clopen sets in $ X $, which exist in abundance, thanks to the zero-dimensionality of $ X $. This has simplified the proof of these results to some extent.
	
 The vertices of the zero-divisor graph $ \gnccx $ are the nonzero zero divisors in the ring $ \ccx $ and two vertices $ f $ and $ g $ are called adjacent, in which case these are connected by an edge if and only if $ fg=0 $. On the other hand the vertices of the comaximal graph $ \gccx $ are the nonzero non-units in $ \ccx $ and two such vertices $ f $ and $ g $ are said to be adjacent if and only if no maximal ideal in $ \ccx $ contains both $ f $ and $ g $. we divide this article into two distinct parts, where the technical sections 2, 3 comprise several facts focussing on the graph $ \gnccx $. The subsequent technical section 4, deals with propositions involving the graph $ \gccx $. The technical section 5 of this paper focusses mainly on the question: when do the two graphs $ \gnccx $ and $ \gccx $ become isomorphic. This is high time we should narrate the organization of our article.
	
	In section 2, we compute the distances between all possible pairs of vertices in $ \gnccx $. It is seen that with the hypothesis $ |X|\geq 3 $, the diameter and girth of $ \gnccx $ are both 3 (Theorem \ref{t-2.5}). It is realized that $ \gnccx $ is triangulated if and only if $ X $ is devoid of any isolated point (Theorem \ref{t-2.8}). It is further realized that $ \gnccx $ is never hypertriangulated (Theorem \ref{t-2.10}). It is also seen that the length of the smallest cycle joining an arbitrary pair of vertices in this graph can be either 3 or 4 or 6 (Theorem \ref{t-2.11}).
	
	In section 3, the problem when does $ \gnccx $ become complemented is mainly addressed. It turns out $ \gnccx $ becomes a complemented graph when and only when the space $ \pccx $ of all minimal prime ideals in $ \ccx $ with Zariski topology is compact (Theorem \ref{t-3.13}). For a typical intermediate ring $ A_c(X) $ lying between $ C_c^*(X) $ and $ \ccx $, its zero-divisor graph $ \gnacx $ is complemented if and only if $ \gnccx $ is complemented (Theorem \ref{t-3.16}). If we combine these two facts, then we get that $ \gnacx $ is complemented if and only if the space $ \pacx $ of all minimal prime ideals in $ \acx $ with Zeriski topology is compact (Theorem \ref{t-3.19}). With the additional hypothesis that $ X $ is strongly zero-dimensional, we realize that the space $ \pccx $ is compact if and only if the space $ \mathcal{P}(C(X)) $ is compact (Theorem \ref{t-3.21}). This is an instance of how tools and results in graph theory can lead to a purely topological result.
	
	In section 4, we take to the technicalities of the comaximal graph $ \gccx $ of the ring $ \ccx $. We compute the girth, diameter, eccentricity of vertices and the length of the possible smallest cycles $ c(f,g) $ joining two vertices $ f $ and $ g $ of this graph. It is proved that $ \gccx $ is triangulated if and only if $ \Gamma_2^{'}(C(X)) $ (the comaximal graph of $ C(X) $) is triangulated and this happens when and only when $ X $ has no isolated point (Remark \ref{r-4.6}). We realize that $\gccx $ is never hypertriangulated. We further examine, the effect on some relevant properties of the graph $ \gccx $ by imposing the additional hypothesis: $ X $ is a $ P $-space/ $ X $ is an almost $ P $-space. We prove that $ \gccx $ is complemented if and only if $ X $ is a $ P $-space (Theorem \ref{t-4.10}). We realize that Rad$ \gccx\equiv$ radius of $ \gccx =3 $ if and only if $ X $ is an almost $ P $-space and does not have any isolated point (Theorem \ref{t-4.2}). We like to mention that these are the countable counterparts of the corresponding facts in the comaximal graph $ \Gamma_2^{'}(C(X)) $ of the ring $ C(X) $ as addressed in the article \cite{ref6}.	For more information on the ring $ \ccx $ and $ C_c^*(X) $, the articles \cite{ref1}, \cite{ref2} , \cite{ref3}, \cite{ref10} are referred.
	
	
	In the last section viz. section 5 of this article, we concentrate mainly on the problem, how far the zero-divisor graph $ \gnccx $ and the comaximal graph $ \gccx $ can be different/ same as graphs. A complete solution to this problem seems to be too wild to venture into. However, we have provided with a large class of zero-dimensional spaces $ X $ for which $ \gccx $ and $ \gnccx $ are non isomorphic (Theorem \ref{to-5.1}). To get more examples of spaces $ X $, having such properties, we initiate a type of quotient graph $ \w{G} $ of any simple graph $ G $. Essentially vertices of $ \w{G} $ are constructed after gluing vertices of $ G $ by some natural rule and then defining the adjacency relation on $ \w{G} $ accordingly. It turns out that if $ X $ is a $ P $-space then the quotient graph of the zero-divisor graph $ \gncx $ of $ C(X) $ is isomorphic to the quotient graph of the comaximal graph $ \gcx $ of $ C(X) $ (Remark \ref{ro-5.7}). We exploit this result to prove with the support of the continuum hypothesis, that a discrete topological space $ X $ is countable if and only if $ \gncx $ is isomorphic to $ \gcx $ if and only if $ \gnccx $ is isomorphic to $ \gccx $ (Theorem \ref{to-5.15}).

	\section{Technical notions related to the graph $ \gnccx $}
For any two vertices $ f,\ g $ in $ \gnccx $, the length of the shortest path containing $ f $ and $ g $ is denoted by $ d(f,g) $. Also $ C(f,g) $ designates 	the length of the smallest cycle joining $ f $ and $ g $. We let $ V_0(C_c(X)) $ stand for the set of vertices in the present zero-divisor graph of $ \ccx $. The number $\sup\{d(f,g): f,g\in \vn\}$ is called the diameter of $ \gnccx $, and for any $ f\in \vn $, $\sup\{d(f,g): g\in \vn\}$ is called the eccentricity of $ f $, and is denoted by $ e(f) $. A vertex with largest eccentricity is called center of the graph and eccentricity of any vertex lying on the center is known as the radius of the original graph. For any $ f\in \ccx $, $ Z(f)=\{x\in X: f(x)=0\} $ is called the zero-set of $ f $. It is proved in Sublemma 1.1 in \cite{ref4} that for a pair of vertices $ f,\ g $ in the zero-divisor graph $ \Gamma(C(X)) $ of $ C(X) $, there exists a vertex $ h $ in the same graph adjacent to both $ f $ and $ g $ if and only if int$Z(f)\ \cap $ int$Z(g)\neq\emptyset $. The following result tells that the countable counterpart of this important fact is true. 
\begin{theorem}\label{t-2.1}
	Given $ f,g\in \vn $, there exists $ h\in\vn $ adjacent to both $ f $ and $ g $ if and only if int$Z(f)\ \cap $ int$Z(g)\neq\emptyset $.
\end{theorem}
	\begin{proof}
	First let there exist $ h\in \vn $ adjacent to both of $ f $ and $ g $ in $ \gnccx $. Then $ h $ is adjacent to each of $ f $ and $ g $ in $ \Gamma(C(X)) $. It follows from sublemma 1.1 in \cite{ref4} that int$Z(f)\ \cap $ int$Z(g)\neq\emptyset $.
	
	Conversely, let int$Z(f)\ \cap $ int$Z(g)\neq\emptyset $. Choose a point $ x\in $ int$Z(f)\ \cap $ int$Z(g)  $. Since $ X $ is zero-dimensional, there exists a clopen set $ W $ in $ X $ such that $ x\in W\subseteq $ int$Z(f)\ \cap $ int$Z(g) $. The characteristic function $ 1_W: X\rightarrow \R $ defined by $ 1_W(y) =1$ if $ y\in W $ and $ 1_W(y) =0$ if $ y\in X\setminus W $ is surely a function in $ \ccx $. We observe that $ 1_W.f=0=1_W.g $. Then $ 1_W\in \vn $ is adjacent to both of $ f $ and $ g$.
	\end{proof}
\begin{theorem}\label{t-2.2}
	The graph $ \gnccx $ is connected meaning that any pair of distinct vertices in this graph can be joined by a path.
\end{theorem}
	\begin{proof}
		Let $ f,g\in \vn $. Then there exist $ h,k\in \vn $ such that $ fh=gk=0$. If there exists $ h\in\vn $ adjacent to both $ f $ and $ g $, then the proof finishes thereon. Assume therefore that no vertex in $ \vn $ is adjacent to both $ f $ and $ g $. It follows from Theorem \ref{t-2.1} that int$Z(f)\ \cap $  int$Z(g)=\emptyset $. Now $ fh=gk=0 $ implies that $ (X\setminus Z(h))\cap (X\setminus Z(k)) \subseteq$ int$Z(f)\ \cap $  int$Z(g)$ and therefore $ (X\setminus Z(h))\cap (X\setminus Z(k))=\emptyset $. This implies that $ hk=0 $. Thus $f\sim h\sim k\sim g$ represents a path joinoing $ f $ and $ g $ in $ \gnccx $. Here the notation $f\sim h  $ stands for: $ f $ is adjacent to $ h $.
	\end{proof}
	\begin{remark}\label{r-2.3}
		The diameter of $ \gnccx \leq 3$.
	\end{remark}
	It is trivial that gr$ \gnccx \equiv$ the girth of $ \ccx \geq3$. We shall show that with the  hypothesis $ |X|\geq3 $, both the diameter and girth of $ \gnccx $ are exactly equal to 3. We need the following proposition for this purpose.
	\begin{theorem}\label{t-2.4}
		Let $ f,g\in \vn $. Then\begin{enumerate}
			\item $ d(f,g)= 1$ if and only if $ Z(f)\cup Z(g)=X $.
			\item $ d(f,g)= 2$ if and only if $ Z(f)\cup Z(g)\neq X $ and int$Z(f)\ \cap $  int$Z(g)\neq \emptyset$.
			\item $ d(f,g)= 3$ if and only if $ Z(f)\cup Z(g)\neq X $ and int$Z(f)\ \cap $  int$Z(g)= \emptyset$.
		\end{enumerate}
	\end{theorem}
	\begin{proof}
		(1) is trivial.
		
		 (2) follows from (1) of the present theorem and Theorem \ref{t-2.1}.
		 
		(3). It follows from (1) and (2) of the present theorem that if $ d(f,g)= 3$ then $ Z(f)\cup Z(g)\neq X $ and int$Z(f)\ \cap $  int$Z(g)= \emptyset$. Conversely, let $ Z(f)\cup Z(g)\neq X $ and int$Z(f)\ \cap $  int$Z(g)= \emptyset$. Then from (1) and (2) of the present theorem $ d(f,g)\neq1 $ and $ d(f,g)\neq2 $. Hence in view of Theorem \ref{t-2.2} and Remark \ref{r-2.3}, we get that $ d(f,g)=3 $.
	\end{proof}
	\begin{theorem}\label{t-2.5}
		Suppose $ |X|\geq3 $. The diameter of $ \gnccx =$ girth of $ \gnccx =3$.
	\end{theorem}
\begin{proof}
	Let $ p,q,r $ be three distinct points in $ X $. Then by using zero-dimensionality of $ X $, we can find out a clopen set $ K $ in $ X $ such that $ \{q,r\}\subseteq K $ and $ p\notin K $. Therefore $ q\notin(X\setminus K)\cup \{r\} $. Using the zero-dimensionality of $ X $ once again, we can produce a clopen set $ L $ in $ X $ such that $ q\notin L $ and $ (X\setminus K)\cup \{r\}\subseteq L $. It is clear that the characteristic functions $ 1_K $ and $ 1_L $ belong to $ \ccx $ and $ Z(1_K)\cup Z(1_L)\neq X $ because $ r\in L\cap K $. On the other hand, $ Z(1_L)=X\setminus L\subseteq K $ and $ Z(1_K)=X\setminus K $ which imply that int$_X Z(1_L)\ \cap $  int$Z(1_K)= \emptyset$. It follows from Theorem \ref{t-2.4}(2) that $ d(1_L, 1_K)=3 $. This implies in view of Remark \ref{r-2.3} that the diameter of $ \gnccx=3 $. Now $ X\setminus L\subseteq K $ implies that $ 1_{X\setminus L}. 1_{X\setminus K}=0 $. We write $ f=1_{X\setminus L}^2+1_{X\setminus K}^2 $. Then $ f $ takes values $ 2 $  on $ (X\setminus L)\cup (X\setminus K) $ and $ 0 $ on $ K\cap L $, a non-empty clopen set in $ X $. It follows that $ f $ is a nonzero divisor of zero in $ \ccx $, i.e., $ f\in \vn $. Therefore there exists $ g\in \vn $ such that $ fg=0 $ and we can take $ g\geq 0 $ on $ X $. Consequently $ g.1_{X\setminus
	 L}=0=g.1_{X\setminus K}=1_{X\setminus
	 L}.1_{X\setminus K} $. Thus $1_{X\setminus
	 L},\ 1_{X\setminus K} $ and $ g $ form the vertices of a triangle. Hence gr$\gnccx =3 $.
\end{proof}
	The following formula is the countable counterpart of Remark 1.5 in \cite{ref4}.
	\begin{theorem}\label{t-2.6}
		Let $ f\in \vn $, then $ e(f)=2 $ if $ X\setminus Z(f) $ is a one member set, otherwise $e(f)=3$.
	\end{theorem}
	\begin{proof}
		First consider the case $ X\setminus Z(f)=\{x\} $. Then the characteristic function $ 1_{\{x\}} $ and $ 1_{X\setminus\{x\}} $ are vertices in $ \gnccx $. Choose $ g\in \vn $, $ g\neq f $. If $ x\in Z(g) $ then $ fg=0 $ which implies that $ d(f,g)=1 $. On the other hand, if $ g $ does not vanish at $ x $ (say $ g=1_{\{x\}} $), then $ Z(g)\subseteq Z(f) $. Consequently, $ Z(f)\cup Z(g)=Z(f)\neq X $ and int$Z(g)\ \cap  $  int$ Z(f)=  $ int$Z(f)\neq \emptyset $. It follows from Theorem \ref{t-2.4}(2) that $ d(f,g)=2 $. Thus $ e(f)=2 $ in this case.
	
	Now consider the other case $ \{x,y\}\subseteq X\setminus Z(f) $, where, $ x\neq y $, $ x,y\in X $. Then by the zero-dimensionality of $ X $, there exists a clopen set $ K $ in $ X $ such that $ K\cap (Z(f)\cup \{y\})=\emptyset $. The function $ 1_{X\setminus K}\in \vn $. We observe that $ Z(f)\cup Z(1_{X\setminus K})\subsetneq X $ and int$ Z(f)\ \cap $ int$ Z(1_{X\setminus K})=\emptyset $. It follows from Theorem \ref{t-2.4}(3) that $ d(f, 1_{X\setminus K})=3 $. Hence $ e(f)=3 $.
\end{proof}	
The following proposition will be very close to determining a topologlical condition on $ X $, equivalent to the graph $ \gnccx $ to be triangulated.
\begin{theorem}\label{t-2.7}
		Let $ f\in\vn $. Then $ f $ is a vertex of a triangle in $ \gnccx $ if and only if int$ Z(f) $ contains at least two distinct points.
\end{theorem}
	\begin{proof}
		First let $ f $ be a vertex of a triangle in $ \gnccx $. Then there exist $ g,h\in \vn $ such that $ fg=gh=hf=0 $. Now $ fg=0 $ implies that $ X\setminus Z(g) \subseteq$ int$ Z(f) $. We argue by contradiction and assume that int$ Z(f)=\{p\} $ for a $ p\in X $. Since $ X\setminus Z(g)\neq\emptyset$, the last inclusion relation therefore implies that int($X\setminus Z(g))=\{p\}  $. Analogously int$ (X\setminus Z(h))=\{p\} $. Therefore int($ X\setminus Z(g)) $ $ \cap $ int($ X\setminus Z(h))\neq\emptyset$. This contradicts the relation $ Z(g)\cup Z(h)=X $. Thus int$ Z(f) $ contains more than one member.
		
	To prove the converse, let $ \{p,q\}\subseteq  $ int$ Z(f)  $ for a pair of distinct points $ p, q $ in $ X $. By zero-dimensionality of $ X $, we can find out a clopen set $ K $ in $ X $ such that $ p\in K\subseteq $ int$ Z(f)\setminus\{q\} $. The function $ 1_K $ is a vertex in $ \gnccx $ and we observe that $ q\in $ int$ Z(f) $ $ \cap $ int$ Z(1_K) $. Thus int$ Z(f)\ \cap $ int$ Z(1_K)\neq\emptyset $. It follows from Theorem \ref{t-2.1} that there exists a vertex $ h $ in $ \gnccx $ adjacent to both $ f $ and $ g $. Consequently, $ f $ is a vertex of a triangle.
	\end{proof}
	\begin{theorem}\label{t-2.8}
		$ \gnccx $ is a triangulated graph if and only if $ X $ does not contain any isolated point. 
	\end{theorem}
	Since the analogous fact for the zero-divisor graph $ \Gamma(C(X)) $ of the ring $ C(X) $ is already established in Proposition 2.1(ii) in \cite{ref4}, we make the following comment.
	\begin{remark}
		For a zero-dimensional space $ X $, the graph $ \gnccx $ is triangulated if and only if the graph $ \Gamma(C(X)) $ is triangulated.
	\end{remark}
	\begin{theorem}\label{t-2.10}
		$ \gnccx $ is never hypertriangulated.
	\end{theorem}
	\begin{proof}
		Since $ X $ is zero-dimensional, there exists a clopen set $ K $ in $ X $ such that $\emptyset\neq K\neq X$. Let $ f=1_K $ and $ g= 1_{X\setminus K}$. Then $ f $ and $ g $ are vertices in $ \gnccx $ and are adjacent because $ Z(f)\cup Z(g)=X $. We observe in addition that int$ Z(f)\ \cap $ int$ Z(g)=\emptyset $. It follows from Theorem \ref{t-2.1} that there does not exist any vertex in $ \gnccx $ adjacent to both $ f $ and $ g $. Hence the edge $ f\sim g $ is never an edge of a triangle. Thus $ \gnccx $ is not hypertriangulated.
	\end{proof}
	 The following proposition is a straightforward countable counterpart of Proposition 2.2 in \cite{ref4}. We simply state this result without any proof because this can be accomplished on using Theorem \ref{t-2.1} and Theorem \ref{t-2.4}.
	 \begin{theorem}\label{t-2.11}
	 	For any two vertices $ f $ and $ g $ in $ \gnccx $,\begin{enumerate}
	 		\item $ c(f,g)=3 $ if and only if $ Z(f)\cup Z(g)=X $ and int$ Z(f)\ \cap $ int$ Z(g)\neq\emptyset $.
	\item $ c(f,g)=4 $ if and only if either $ Z(f)\cup Z(g)\neq X $ and int$ Z(f)\ \cap $ int$ Z(g)\neq\emptyset $ or $ Z(f)\cup Z(g)=X $ and int$ Z(f)\ \cap $ int$ Z(g)=\emptyset $.
	\item $ c(f,g)=6 $ if and only if $ Z(f)\cup Z(g)\neq X $ and int$ Z(f)\ \cap $ int$ Z(g)=\emptyset $.
 	\end{enumerate}
	 \end{theorem}
	
	\section{When does $ \gnccx $ become complemented?}
	Like any graph, we call $ \gnccx $ a complemented one, if for any vertex $ f $, there exists a vertex $ g $, orthogonal to $ f $, and we write $ f\perp g $ in the sense that $ fg=0 $ and there does not exist any vertex $ h $, adjacent to both $ f $ and $ g $. The next result follows from Theorem \ref{t-2.1} and Theorem \ref{t-2.4}(1).
	\begin{theorem}\label{t-3.1}
		$ \gnccx $ is a complemented graph if and only if for any vertex $ f $, there exists a vertex $ g $ such that $ Z(f)\cup Z(g)=X $ and int$ Z(f)\ \cap $ int$ Z(g)=\emptyset $.
	\end{theorem}
	At this point we like to mention that it is proved in Corollary 2.5 in \cite{ref4} that the zero-divisor graph $ \Gamma(C(X)) $ of the ring $ C(X) $ is complemented if and only if the space of minimal prime ideals in the ring $ C(X) $ is compact. We shall obtain a countable counterpart of this fact and also some other related results concerning the zero-divisor graphs of the intermediate rings that lie between the two rings $ C_c^*(X) $ and $ C_c(X) $. We need to recall a few basic information on the space of minimal prime ideals of a commutative ring $ A $ from the article \cite{ref9} and specialize these results with the choice $ A=\ccx $.

	\begin{notation}
		 Let $ \pccx $ be the set of all minimal prime ideals in $ \ccx $. Each prime ideal in $ \ccx $ contains a minimal prime ideal (typical Zorn's Lemma arguments) and there are enough of prime ideals in $ \ccx $, indeed, each maximal ideal in $ \ccx $ is a prime ideal. Thus $ \pccx \neq\emptyset $.
	\end{notation}
\begin{notation}
	For any $ S\subseteq\ccx $ and $ \mathscr{S}\subseteq\pccx $, we set $ h(S)=\{P\in \pccx: P\supseteq S\} $, called the hull of $ S $ and $ k(\mathscr{S})=\bigcap\{P:P\in \mathscr{S}\} $, called the kernel of $ \mathscr{S} $. It is absolutely a routine check to prove that $ \mathcal{B}=\{h(f):f\in \ccx\} $ is a base for the closed sets of some topology on $ \pccx $. We call the set $ \pccx $ equipped with this topology, the space of minimal prime ideals in $ \ccx $.
\end{notation}	
	We reproduce below the following 4 facts from \cite{ref9} with the choice $ A=\ccx $. We will need these facts to prove the main results in this section.
	\begin{theorem}\label{t-3.5}
		Let $ f,g\in \ccx $,\ $ S_1, S_2 $ are subsets of $ \ccx $ and $ \mathscr{S}_1, \mathscr{S}_2 $ are subfamilies of $ \pccx $. Then
		\begin{enumerate}
			\item If $ S_1\subseteq S_2 $ then $ h(S_1)\supseteq h(S_2) $.
			\item If $ \mathscr{S}_1\subseteq\mathscr{S}_2 $ then $ k(\mathscr{S}_1)\supseteq k(\mathscr{S}_2) $.
				\item $h( S_1\cap S_2) = h(S_1)\cup h(S_2) $ if $S_1$ and $S_2$ are ideals in $\ccx$.
			\item $ k(\mathscr{S}_1\cup \mathscr{S}_2)=k(\mathscr{S}_1)\cap k(\mathscr{S}_2) $.
		\end{enumerate}
	Indeed, these results can be established independently almost immediately.
	\end{theorem}

	For any $ S\subseteq\ccx $, set $ Ann(S)=\{f\in \ccx: fg=0$ for each $g\in S\} $.
\begin{theorem}\label{t-3.6}
	For any $ f\in\ccx $, $ h(Ann(f))=\pccx\setminus h(f) $. It follows that $ h(f) $ and $ h(Ann(f)) $ are disjoint clopen sets in $ \pccx $.
\end{theorem}
	\begin{theorem}\label{t-3.7}
		Let $ f,g,l\in\ccx $ and $ S\subseteq\ccx $. Then the following results hold.
		\begin{enumerate}
			\item $ kh(Ann(S))=Ann(S) $.
			\item $ Ann(l)=Ann(f)\cap Ann(g) $ if and only if $ h(l)=h(f)\cap h(g) $.
			\item $ Ann(Ann(f))=Ann(g) $ if and only if $ h(f)=h(Ann(g)) $.
		\end{enumerate}
	\end{theorem}
	The following proposition is the countable counterpart of Lemma 5.4 in \cite{ref9}.
	\begin{theorem}\label{t-3.9}
		Let $ f,g\in \ccx $, then\begin{enumerate}
			\item $ h(Ann(f))\subseteq h(g) $ if and only if $ Z(f)\cup Z(g)=X $.
			\item $ h(Ann(f))\supseteq h(g) $ if and only if int$Z(f)\ \cap  $ int$ Z(g)=\emptyset $.
		\end{enumerate}
	\end{theorem}
	\begin{proof}
		(1). Let $ h(Ann(f))\subseteq h(g) $. Then from Theorem \ref{t-3.5}(2) and \ref{t-3.7}(1), it follows that $ Ann(f)\supseteq kh(g) $. But $ g\in kh(g) $ implies that $ fg=0 $. On the other hand if $ Z(f)\cup Z(g)=X $, then $ g\in Ann(f) $. Consequently $ h(g)\supseteq h(Ann(f)) $ by Theorem \ref{t-3.5}(1).
		
		(2). In view of Theorem \ref{t-3.6},
		\begin{alignat*}{2}
&		  \ h(g)\subseteq h(Ann(f))\\ 
		 \Leftrightarrow&\  h(g)\cap h(f)=\emptyset\\
		 \Leftrightarrow&\  \pccx\setminus h(Ann(g))\cap (\pccx\setminus h(Ann(f)))=\emptyset \\
		 \Leftrightarrow &\  h(Ann(f))\cup h(Ann(g))=\pccx\\
		 \Leftrightarrow &\ kh(Ann(f))\cap kh(Ann(g))=k(\pccx)=\{0\}\\
		&\ \ \ \ \ \ \ \ \ \ \ \ \ \ \ \ \ \ \ \ \ \ \ \ \ \ \ \ \ \ \ \  (\text{by Theorem \ref{t-3.5}(3),\ \ \ref{t-3.5}(4)\ \ and \ref{t-3.7}(1))}\\
		 \Leftrightarrow &\  Ann(f)\cap Ann(g)=\{0\}\ (\text{by Theorem \ref{t-3.7}(1))}\\
		 \Leftrightarrow &\ Ann(f^2+g^2)=\{0\} \\
		 \Leftrightarrow &\  (f^2+g^2)\  \text{is not a divisor of zero in}\ \ccx\\
		 \Leftrightarrow &\ int Z(f^2+g^2)=\emptyset\\
		 \Leftrightarrow &\ intZ(f)\ \cap\ intZ(g)=\emptyset.
		\end{alignat*}
	\end{proof}
Theorem \ref{t-3.1} in conjunction with Theorem \ref{t-3.9} yields the following result:
	\begin{theorem}\label{t-3.10}
		The zero-divisor graph $ \gnccx $ is complemented if and only if for any vertex $ f $ in this graph, there exists a vertex $ g $ such that $ h(g)=h(Ann(f)) $.
	\end{theorem}

		\begin{definition}
		A commutative ring $ A $ without nonzero nilpotents is
		said to satisfy the annihilator condition, or is called an a.c. ring, if for every $ x,y\in A $,
		there exists $ z\in A $ such that \[ Ann(z)=Ann(x)\cap Ann(y). \] Since for every $f,g\in \ccx$, $ Ann(h)=Ann(f)\cap Ann(g)$ where $h=f^2+g^2\in \ccx$, thus $ \ccx $ satisfies the annihilator condition, furthermore, $\ccx$ contains no nonzero nilpotents hence the following proposition is immediate by Theorem 3.4 in \cite{ref9} and Theorem \ref{t-3.7}(3).
		
	\end{definition}
	
	\begin{theorem}\label{tp-2.18}
		The space $ \pccx  $ is compact if and only if for every $ f\in\ccx  $ there exists $ g\in \ccx   $ such that $ h(g)=h(Ann(f)) $.
	\end{theorem}

	We can combine Theorem \ref{t-3.10} and Theorem \ref{tp-2.18} to get the following result.
	\begin{theorem}\label{t-3.13}
		$ \gnccx $ is a complemented zero-divisor graph if and only if the space $ \pccx $ is compact.
	\end{theorem}
	\begin{proof}
		First let $ \gnccx $ be complemented. Choose $ f\in \ccx $.\\
		Case1. Let $ f $ be a divisor of zero in $ \ccx $. Then by Theorem \ref{t-3.10}, there exists $ g\in \ccx $ such that $ h(g)=h(Ann(f)) $.\\
		Case2. Let $ f $ be not a divisor of zero in $ \ccx $. Then $ Ann(f)=\{0\} $. Consequently, $ h(Ann(f))=\pccx=h(g) $ on taking $ g=0 $.
		
		Conversely, let $ \pccx $ be compact and $ f $ be a vertex in $ \gnccx $. It follows from Theorem \ref{tp-2.18} and Theorem \ref{t-3.9} that there exists $ g\in \ccx $ such that $Z(f)\cup Z(g)=X $ and int$Z(f)\ \cap  $ int$ Z(g)= \emptyset$. We assert that $ g$ is a vertex in $\gnccx $. If not, then $ g=0 $ which implies that int$Z(f)\ \cap  $ int$ Z(g)= $ int$ Z(f)\neq\emptyset $, this is a contradiction. By Theorem \ref{t-3.1}, $ \gnccx $ is complemented.
		\end{proof}
	
	Let $ \acx $ be a ring lying between the rings $ C_c^*(X)$ and $ C_c(X) $. We call $ \acx $ a typical intermediate ring in our present study. Therefore the zero-divisor graph $ \gnacx $ considered over the ring $ \acx $ in the usual manner is a subgraph of the zero-divisor graph $ \gnccx $ constructed over $ \ccx $. The following simple result will be helpful to us in the study of $ \acx $ vis-a-vis $ \ccx $.
	\begin{theorem}\label{t-3.14}
		Given $ f\in \ccx $, there exists a positive unit $ u $ in the ring $ \ccx $ such that $ uf\in C_c^*(X) $.
	\end{theorem}
	\begin{proof}
		Choose $ u=\frac{1}{1+|f|} $.
	\end{proof}
A close look into the proof of Theorem \ref{t-2.1} and Theorem \ref{t-2.4}(1) reveals that for any two vertices $ f,g\in \acx $, int$ Z(f)\ \cap $ int$ Z(g)\neq\emptyset$ if and only if $ f $ and $ g $ admit a vertex in $ \acx $ adjacent to both of $ f $ and $ g $, and $ fg=0 $ if and only if $ f $ and $ g $ are adjacent in the graph $ \gnacx$. Therefore as in the proof of Theorem \ref{t-3.1} we can write the following sharpened version of this theorem.
\begin{theorem}\label{t-3.15}
	$ \gnacx $ is a complemented graph if and only if given a vertex $ f $ in this graph, there exists a vertex $ g $ in the same graph such that $ Z(f)\cup Z(g)=X $ and int$ Z(f)\ \cap $ int$ Z(g)=\emptyset$.
\end{theorem}
	\begin{theorem}\label{t-3.16}
		$ \gnccx $ is a complemented graph if and only if $ \gnacx $ is complemented.
	\end{theorem}
	\begin{proof}
		First assume that $ \gnccx $ is complemented. Let $ f $ be a vertex in $ \gnacx $. Then there exists a vertex $  g $ in $ \gnccx $ such that $ Z(f)\cup Z(g)=X $ and int$ Z(f)\ \cap $ int$ Z(g)=\emptyset$ (this follows from Theorem \ref{t-3.1}). Now we apply Theorem \ref{t-3.14}, to find out a unit $ u $ in $ \ccx $ such that $ ug\in C_c^*(X) $ and hence $ ug\in A_c(X) $. It is clear that $ ug $ is a vertex in $ \acx $ and $ Z(ug)=Z(g) $. This yields $ Z(f)\cup Z(ug)=X $ and int$ Z(f)\ \cap $ int$ Z(ug)=\emptyset$. We now apply Theorem \ref{t-3.15} to conclude that $ \gnacx $ is complemented.
		
		Conversely let $ \gnacx $ be complemented. Suppose $ f $ is a vertex in $ \gnccx
		 $. It follows from Theorem \ref{t-3.14} that there exists a unit $ u $ in $ \ccx $ such that $ uf $ is a vertex in $ \gnacx $. Hence there exists a vertex $ h\in \gnacx $ such that $ Z(uf)\cup Z(h)=X $ and int$ Z(uf)\ \cap $ int$ Z(h)=\emptyset$ (this follows from Theorem \ref{t-3.15}). We get $ Z(f)\cup Z(h)=X $ and int$ Z(f)\ \cap $ int$ Z(h)=\emptyset$ and clearly $ h $ is a vertex in $ \gnccx $. Hence from Theorem \ref{t-3.1}, we get that $ \gnccx $ is complemented.
	\end{proof}
\begin{remark}
	By closely following the above arguments we can realize that for a Tychonoff space $ X $ (not necessarily zero-dimensional) and any ring $ A(X) $ lying between $ C^*(X) $ and $ C(X) $, $ \Gamma (C(X)) $ is complemented if and only if $ \Gamma(A(X)) $ is complemented.
\end{remark}
	Before proceding further we reproduce the following proposition (Theorem 5.1 in \cite{ref9}), which we will need to prove a theorem concerning the minimal prime ideals $ \pacx $ of a typical intermediate ring $ \acx $ lying between $ C_c^*(X) $ and $ C_c(X) $.
	\begin{theorem}\label{t-3.17}
		Let $ {B} $ be a commutative reduced ring and $ A $ a subring of it with the following condition: for each $ b\in B $, there exists $ a_b \in A$ and $ u_b\in B $ such that $ b=a_bu_b $ and Ann$(u_b)=\{0\}  $. Then the spaces $ \mathcal{P}(B) $ and $\mathcal{P}(A)$ of minimal prime ideals of the rings $ B $ and $ A $ become homeomorphic under the map: \begin{alignat*}{2}
&  \mathcal{P}(B)\rightarrow \mathcal{P}(A)\\
&\ P\mapsto P\cap A
		\end{alignat*}
	\end{theorem}
We use this theorem to prove the next result:
\begin{theorem}\label{t-3.18}
	The space $ \pccx $ of all minimal prime ideals in $ \ccx $ is homeomorphic to the space $ \pacx $ of all minimal prime ideals in $ \acx $.\
	
\end{theorem}
	\begin{proof}
		It follows from Theorem \ref{t-3.14} that given $ f\in \ccx $, the function u=$\frac{1}{1+|f|}$ is a unit in $ \ccx $ for which $ uf \in \acx$. Consequently, we can write $ f=uf(1+|f|) $. We note that Ann$ (1+|f|)=\{0\} $. A straight way application of Theorem \ref{t-3.17} now yields that $ \pacx $ is homeomorphic to $ \pccx $.

	\end{proof}
	On combining Theorem \ref{t-3.13}, Theorem \ref{t-3.16} and Theorem \ref{t-3.18}, we get the following proposition.
	
	\begin{theorem}\label{t-3.19}
		The zero-divisor graph $ \gnacx $ of the intermediate ring $ \acx $ is complemented if and only if the space $ \pacx $ is compact.
	\end{theorem}

\begin{remark}
	On making some modifications in the above chain of arguments needed to prove Theorem \ref{t-3.19} and taking into consideration Corollary 2.5 in \cite{ref4} the following fact comes out for a Tychonoff space $ X $ not necessaryly zero-dimensional: the zero-dimensional graph $ \Gamma(A(X)) $ of a ring $ A(X) $ lying between $ C^*(X) $ and $ C(X) $ is complemented if and only if the sapce $ \mathcal{P}(A(X)) $ is compact. 
\end{remark}

	It is already remarked in \cite{ref4}, vide the comments proceeding Corollary 2.5 in \cite{ref4} that for a Tychonoff space $ X $, $ \Gamma(C(X)) $ is complemented if and only if for every vertex $ f $ in $ C(X) $ there exists a vertex $ g $ in $ C(X) $ with $ Z(f)\cup Z(g)=X $ and int$ Z(f)\ \cap $ int$ Z(g)=\emptyset$. Incidentally, the countable counterpart of this result on replacing $ C(X) $ by $ C_c(X) $ with $ X $, a zero-dimensional space is proved in Theorem \ref{t-3.1} of the present article. Since for a strongly zero-dimensional space $ X $, i.e., for which $ \beta X $ is zero-dimensional, $ \{Z(f):f\in C_c(X)\}=\{Z(f):f\in C(X)\}  $ (Theorem 2.4, \cite{ref3}), we therefore get the following result.
	\begin{theorem}\label{t-3.20}
		For a strongly zero-dimensional space $ X $, the zero-divisor graph $ \gnccx $ is complemented if and only if the zero-divisor graph $ \Gamma(C(X)) $ is complemented.
	\end{theorem}
	
	On combining Theorem \ref{t-3.13} of the present article and Corollary 2.5 in \cite{ref4}, the following fact therefore comes out.
	\begin{theorem}
		For a strongly zero-dimensional space $X$, the following two statements are equivalent.
	\begin{enumerate}\label{t-3.21}
		\item 	The space $ \pccx $ of minimal prime ideals in $ \ccx $ is compact.\item The space $ \mathcal{P}(C(X)) $ of minimal prime ideals in $ C(X) $ is compact.
	\end{enumerate}
	\end{theorem}
	Now let us examine what happens to the space $ \pccx $ and $ \mathcal{P}(C(X)) $ with the choice $ X= $ a $ P $-space. With such a special choice of $ X $, it follows from Proposition 5.3 in \cite{KGN} that $ \ccx $ is a Von-Neumann regular ring and we get from $ 4J $ in \cite{ref8} that $ C(X) $ is also a Von-Neumann regular ring. Consequently, $ \pccx $ coincides with the set $ \mathcal{M}_c(X) $ of all maximal ideals in $ \ccx $ with hull-kernel topology and therefore $ \pccx=\beta_0X $, the Banaschewski compactification of $ X $ \cite{ref3}. By an identical reasoning $ \mathcal{P}(C(X))=\beta X  $, the Stone-$\check{C}$ech compactification of $ X $. Since $ X $ is strongly zero-dimensional as it is a $ P $-space, it follows that $ \beta X=\beta_0X $.
	
	Thus for this special choice of $ X $ viz. that $ X $ is a $ P $-space we can say that $ \pccx $ and $ \mathcal{P}(C(X)) $ are homeomorphic  spaces. We feel it therefore natural to ask the following question.
	\begin{question}
		Are the two spaces $ \pccx $ and $ \mathcal{P}(C(X)) $ homeomorphic for an arbitrary strongly zero-dimensional space $ X $?
	\end{question}
	\section{Comaximal graph $ \gccx $ associated with the ring $ \ccx $.}
We recall that the vertices of the graph $ \gccx $ are just the nonzero non-units in the ring $ \ccx $. Thus an $ f\neq0 $ in $ \ccx $ is a vertex of this graph if and only if  $ Z(f)\neq\emptyset $. Let $ V_2(\ccx) $ stand for the set of vertices in this graph. Two functions $ f $ and $ g $ in $ \vtccx $ are adjacent if and only if $ <f>+<g>=C(X) $, here $ <f> $ is the principal ideal in $ \ccx $ generated by $ f $. Surely $ f $ and $ g $ are adjacent when and only when $ Z(f)\cap Z(g)=\emptyset $.
\begin{theorem}\label{t-4.1}
	The graph $ \gccx $ is connected.
\end{theorem}
	\begin{proof}
		Let $ f\in \vtccx $. Then there exist $ x,y\in X $ such that $ f(x)=0 $ and $ f(y)\neq 0 $. Since $ X $ is zero dimensional, there exists a clopen set $ W $ in $ X $ such that $ Z(f)\subseteq W $ and $ y\notin W $. The characteristic function $ 1_W \in \vtccx$ and $ Z(f)\cap Z(1_W) =\emptyset$. This shows that $ 1_W $ and $ f $ are adjacent vertices in this graph.
	\end{proof}
	\begin{theorem}\label{t-4.2}
		Let $ f,g\in \vtccx $, then\begin{enumerate}
			\item $ d(f,g)=1 $ if and only if $ Z(f)\cap Z(g)=\emptyset $.
	\item $ d(f,g)=2 $ if and only if $ Z(f)\cap Z(g)\neq\emptyset $ and $ Z(f)\cup Z(g)\neq X$.
	\item $ d(f,g)=3 $ if and only if $ Z(f)\cap Z(g)\neq\emptyset $ and $ Z(f)\cup Z(g)= X $.
	\end{enumerate}
	\end{theorem}
	The proof of this theorem is analogous to that of Lemma 2.1 in \cite{ref6} and is therefore omitted.
	\begin{theorem}\label{t-4.3}
		Diameter of $ \gccx=3 $ if and only if $ X $ contains at least $ 3 $ distinct points.
	\end{theorem}
	\begin{proof}
		First assume that $ X $ contains at least $ 3 $ distinct members $ x,y,z $. Then by the zero-dimensionality of $ X $, there exists a clopen set $ K $ in $ X $ such that $ \{x,y\}\subseteq K $ and $ z\notin K $. We use the zero-dimensionality of $ X $ once again to produce a clopen set $ L $ in $ X $ with the property $ (X\setminus K)\cup \{x\} \subseteq L $ and $ y\notin L $.
		
		The characteristic function $ 1_{X\setminus K} $ and $ 1_{X\setminus L} $ are functions in $ \ccx $. We observe that $ Z(1_{X\setminus K}) \cap Z(1_{X\setminus L})\neq\emptyset$ because $ K\cap L\neq\emptyset $ as $ x\in K\cap L $. On the other hand $  Z(1_{X\setminus K}) \cup Z(1_{X\setminus L})= K\cup L=X$. It follows from Theorem \ref{t-4.2}(3) that $ d( 1_{X\setminus K}, 1_{X\setminus L})=3 $. Thus diam$ \gccx=3 $.
		
		If $ X $ contains just a single point then $ \ccx $ is isomorphic to the field $ \R $ and therefore the vertex set $ \vtccx=\emptyset $ and $ \gccx $ becomes an empty graph. Next let $ X=\{a,b\} $, a two membered set. Then $ C(X)=\ccx $ becomes isomorphic to the ring $ \R\times \R $, which is the direct product of the field $ \R $ with itself. The vertices in the comaximal graph $ \gccx $ can therefore be identified with elements of the form $ (r,0) $ or $ (0,s) $, $ r,s\in\R $, $ r\neq0 $, $ s\neq 0 $. Therefore for any two distinct vertices $ f $ and $ g $ in this graph either $ Z(f)\cap Z(g)=\emptyset $ or the two relations  $ Z(f)\cap Z(g)\neq\emptyset $ and $ Z(f)\cup Z(g)\subsetneq X $ combined together. Consequently from Theorem \ref{t-4.2} we get $ d(f,g)=1 $ or $ d(f,g)=2 $. Also for any vertex $ f $, $ d(f,2f)=2$, an easy verification. Hence diam$ \gccx=2 $.
	\end{proof}
	The following result is the countable counterpart of Lemma 3.1 in \cite{ref6}.
	\begin{theorem}\label{t-4.4}
		An $ f\in\gccx $ is a vertex of a triangle when and only when Coz$f=X\setminus Z(f)  $ contains at least two distinct members.
	\end{theorem}
	\begin{proof}
		By closely following the arguments in the proof of ``$ \Leftarrow $'' part of Lemma 3.1 in \cite{ref6}, it is not hard to prove that if $ f $ is the vertex of a triangle then Coz$ f $ is not a one membered set.
		To prove the converse part of this theorem, let Coz$ f\supseteq\{p,q\} $ where $ p\neq q $, $ p,q\in X $. Then by the zero-dimensionality of $ X $, there exists a clopen set $ K $ in $ X $ such that $ q\notin K $ and $ \{p\}\cup Z(f) \subseteq K$. Then $\chi_K\in \ccx$ and is a vertex of $ \gccx $. It is clear that $ p\notin Z(f)\cup Z(\chi_K) $. We use the zero-dimensionality of $ X $ once again to find out a clopen set $ L $ in $ X $ with the property: $ Z(f)\cup Z(\chi_K) \subseteq L$ and $ p\notin L $. Then $ \chi_K\in \ccx $ and is a vertex of $ \gccx $. We check that $ Z(f)\cap Z(\chi_K)=Z(\chi_K)\cap Z(\chi_L)=Z(\chi_L)\cap Z(f)=\emptyset $.  Thus the vertices $ f,\ \chi_K,\ \chi_L $ make a triangle.
	\end{proof}
	\begin{theorem}\label{t-4.5}
		$ \gccx $ is triangulated if and only if $ X $ has no isolated point.
	\end{theorem}
Theorem 5.1 in \cite{ref6} says that a Tychonoff space $ X $ is devoid of any isolated point if and only if the comaximal graph $ \gccx $ is triangulated. Thus we can make the following comment.
	\begin{remark}\label{r-4.6}
		A zero-dimensional space $ X $ is devoid of any isolated point if and only if $  \Gamma_2^{'}(C(X)) $ is triangulated if and only if $\gccx $ is triangulated.
	\end{remark}
If $ X $ contains at least three distinct points $ p,q,r $ then on using the zero-dimensionality of $ X $, we can find a clopen set $ K $ in $ X $ such that $ r\notin K $ and $ \{p,q\}\subseteq K $. The function $ \chi_K $ is a vertex of $ \gccx $ and Coz$(\chi_K)\supseteq \{p,q\}  $. This yields to the following fact (we use Theorem \ref{t-4.4} for this purpose).
\begin{remark}
If $ |X|\geq 3 $, then girth of $ \gccx=3 $.
\end{remark}
The following proposition is the countable counterpart of Lemma 4.1 in \cite{ref6}. We simply enunciate it, because its proof can be accomplished by arguing analogously as in the proof of Lemma 4.1 \cite{ref6} and taking care of the zero-dimensionality of $ X $.
\begin{theorem}\label{t-4.8}
	Let $ f,g\in \vtccx $.
\begin{enumerate}
	\item Suppose $ Z(f)\cap Z(g)=\emptyset$, then
		\begin{enumerate}
		\item $ Z(f)\cup Z(g) \neq X$ if and only if $ C(f,g)=3 $.
		\item $ Z(f)\cup Z(g) =X$ if and only if $ C(f,g)=4 $.
	\end{enumerate}
\item Let $ Z(f)\cap Z(g)\neq \emptyset $, then
\begin{enumerate}
	\item $ Z(f)\cup Z(g)\neq X $ if and only if $ C(f,g)=4 $.
	\item $ Z(f)\cup Z(g)=X $ if and only if $ C(f,g)=6 $.
\end{enumerate}
\end{enumerate}
\end{theorem}
	Now for the zero-dimensional space $ X $, there exists a clopen set $ K $ in $ X $ such that $ K\neq \emptyset $ and $ K\neq X $. Clearly then $ 1_K $ and $ 1_{X\setminus K} \in \vtccx$ and $ Z(1_K)\cap Z(1_{X\setminus K}) =\emptyset$ and $ Z(1_K)\cup Z(1_{X\setminus K}) =X$. This yields on applying Theorem \ref{t-4.8}(1)(b) above that $ C(1_K, 1_{X\setminus K}) =4$.
	\begin{remark}
		 $ \gccx $ is never hypertriangulated.
	\end{remark}
	\begin{theorem}\label{t-4.10}
		The following statements are equivalent for a zero-dimensional space $ X $.
		\begin{enumerate}
			\item The comaximal graph $ \gccx $ is complemented.
			\item The comaximal graph $\gcx  $ is complemented.
			\item $ X $ is a $ P $-space.
		\end{enumerate}
	\end{theorem}
	\begin{proof}
	$ (1)\Rightarrow (3)$. Assume that $ \gccx $ is complemented. Choose $ f\in \ccx $. If $ f=0 $ or a unit in $ \ccx $, then $ Z(f)=X $ or $ Z(f)=\emptyset $. Assume that $ f $ is a nonzero non-unit in $ \ccx $. Then $ f $ is a vertex in $ \gccx $. The hypothesis indicates that there is a vertex $ g $ in $ \gccx $ such that $ f\perp g $. Consequently $ Z(f)\cap Z(g)=\emptyset $, by Theorem \ref{t-4.2}(1) and also $ C(f,g)>3 $. It follows from Theorem \ref{t-4.8}(a) that $ Z(f)\cup Z(g)=X$. Thus $ Z(f) $ is a clopen set in $ X $. Therefore each zero-set in $ X $ is a clopen set, hence $X$ is a $P$-space.
	
	$ (3)\Leftrightarrow (2) $ follows from Theorem 5.3 in \cite{ref6}.
	
	$ (3)\Rightarrow (1) $: Let $ X $ be a $ P $-space and $ f $ a vertex in $ \gccx $. Then $ Z(f) $ is a non empty proper clopen subset of the space $ X $. Consequently $ 1_{Z(f)} $ is a vertex in $ \gccx $ and $ Z(f)\cap Z(1_{Z(f)}) =\emptyset$ and also $ Z(f)\cup Z(1_{Z(f)}) =X$. It follows from Theorem \ref{t-4.8}(1)(b) that $ C(f,1_{Z(f)})=4 $. Thus $ 1_{Z(f)} $ is a vertex orthogonal to $ f $ in $ \gccx $. Hence $ \gccx $ is complemented.
	\end{proof}
	Before  presenting the last theorem of this section  which is the countable counterpart of Theorem 5.6 in \cite{ref6} we need to establish the following elementary fact about the nature of the fixed maximal ideals in $ \ccx $.
	\begin{theorem}\label{t-4.11}
		A fixed maximal ideal $ M_p=\{f\in \ccx: f(p)=0\} $, $ p\in X $ is principal in $ \ccx $ if and only if $ p $ is an isolated point of $ X $.
	\end{theorem}
	\begin{proof}
		It can be easily proved that by using first isomorphism theorem of Algebra that $ \{M_p:p\in X \}$ is the entire list of fixed maximal ideals in $ \ccx $, here $ M_p=\{f\in \ccx: f(p)=0\} $. First assume that $ M_p=<f> $ for some $ f\in \ccx $. Thus $ p\in Z(f) $. We claim that $ Z(f)=\{p\} $, for if $ q\in Z(f) $ for some $ q\neq p $ in $X $ then $ f\in M_q$ which implies that $ M_p\subseteq M_q $ and hence $ p=q $, due to the maximality of the ideals, a contradiction. Now the function $ f^{\frac{1}{5}} \in\ccx$ and $ f^{\frac{1}{5}}(p)=0 $ and therefore $ f^{\frac{1}{5}}\in <f> $. Consequently there exists $ h\in \ccx $ such that $ f^{\frac{1}{5}}=hf $ and hence $ h(x)=\frac{1}{f^{\frac{4}{5}}}(x) $ for all $x\in  X\setminus Z(f) $. We now assert that $ p $ is an isoleted point of $ X $. If possible let $ p $ be a nonisolated point in $ X $. Then each neighbourhood of $ p $ contains infinitely many points of $ X $. Again since $ h $ is continuous at the point $ p $, there  exists a $ \delta>0 $ such that $ |h(x)|\leq\delta $ for all $ x $ belonging to a neighbourhood $ U $ of $ p $ in $ X $ and hence$ |f^{\frac{4}{5}}(x)|\geq\frac{1}{\delta} $ for all $ x\in U\setminus \{p\} $. But since $ f^{\frac{4}{5}}(x)=0 $, there exists a neighbourhood $ V $ of $ p $ in $ X $ such that $ |f^{\frac{4}{5}}(x)|<\frac{1}{\delta} $ for all $ x\in (U\cap V)\setminus\{p\} $, a contradiction. Thus $ p $ is an isolated point of $ X $.
		
		Conversely, let $ p $ be an isolated point of $ X $. Then  the function $ \chi_{X\setminus\{p\}}\in \ccx $ and $ Z(\chi_{X\setminus\{p\}})=\{p\} $. It follows that for any $ g\in M_p $, $ Z(g) $ is a neighbourhood of $ Z(\chi_{X\setminus\{p\}}) $. Hence $ g $ is a multiple of $ \chi_{X\setminus\{p\}} $ in the ring $ \ccx $. This proves that $ M_p=(\chi_{X\setminus\{p\}})= $ the principal ideal in $ \ccx  $ generated by $ \chi_{X\setminus\{p\}} $.
	\end{proof}
	\begin{theorem}\label{t-4.12}
		The following statements are equivalent for a zero-dimensional space $ X $.\begin{enumerate}
			\item $ X $ is an almost $ P $-space $ ( $meaning that every non empty zero-set in $ X $ has non  empty interior$ ) $ and there is no isolated point in $ X $.
\item $ \ccx $ is an almost regular ring $ ( $meaning that every nonzero non-unit in $ \ccx $ is a divisor of zero$ ) $ and there does not exist any maximal ideal in $ \ccx $ which is principal.
\item Radius of $ \gccx=3 $.
\item For each vertex $ f $ in $ \gccx,$ there exists a vertex $ g $ in $ \gccx $ such that $ C(f,g)=6 $.
		\end{enumerate}
	\end{theorem}
\begin{proof}
	Since a zero-dimensional space is almost $ P $ if and only if $ \ccx $ is an almost regular ring (vide: Theorem 4.6 in \cite{ref1}), it follows by taking into consideration the result of Theorem \ref{t-4.11} that $ (1)\Leftrightarrow (2) $.

$ (1)\Rightarrow (3)$. Let $ f\in\vtccx $. It is sufficient to show that $ e(f)=3 $. Indeed $ f $ is a nonzero non-unit in $ \ccx $ implies that $ Z(f)\neq\emptyset $, which further yields because $ X $ is almost $ P $ that int$ Z(f)\neq\emptyset $ and int$ Z(f) $ contains at least two distinct points. We argue by contradiction. If possible let $ e(f)\neq3 $. It follows that for each $ g\in \vtccx$, $ d(f,g)<3 $. This implies in view of Theorem \ref{t-4.2} that $ \forall\ g\in \vtccx $, $ Z(f)\cup Z(g)=X\Rightarrow Z(f)\cap Z(g)=\emptyset $. This is the case when and only when $ \forall\ g\in \vtccx $, Coz$ (g)\subseteq Z(f)\Rightarrow Z(f)\subseteq$ Coz$ (g) $ and this happens if and only if $ \forall\ g\in \vtccx $, Coz$ (g)\subseteq Z(f)\Rightarrow Z(f)=$ Coz$ (g) $...$ (*) $. Now using the zero-dimensionality of $ X $, we can produce a $ g\in\ccx $ such that $ \emptyset\ \neq $Coz$ (g) \subseteq$ int$ Z(f)\subseteq Z(f)\subsetneq X $. Surely then $ g $ is a nonzero non-unit in $ \ccx $ and  therefore $ g\in\vtccx $. It follows from the above relation $ (*) $ that Coz$ (g)=Z(f) $ and hence $ Z(f) $ is clopen in $ X $. To get the desired contradiction, it suffices to show that Coz$ (g) $ is a singleton. If possible let there exist $ p,q\in  $Coz$ (g) $, $ p\neq q $. Then on using zero-dimensionality of $ X $ once again, we can find out an $ h\in \ccx $ such that $ h($Coz$(f)\cup \{q\})=\{0\} $and $ h(p)=1 $. This implies that Coz$ (h) \subseteq Z(f)$. But $ q\in Z(f) $ and $ q\notin $Coz$ (h) $ indicates that Coz$ (h)\subsetneq Z(f) $. This contradicts the relation $ (*) $ above.

$ (3)\Rightarrow (1) $. Suppose $ f\in C(X) $ with $ Z(f)\neq\emptyset $. Then by Lemma 2.2(c) in \cite{ref3} there exists $ g\in \ccx $ such that $ \emptyset\neq Z(g)\subseteq Z(f) $. By $ (3) $ it follows that $ e(g)=3 $. So  there exists $ h\in \vtccx $ such that $ d(g,h)=3 $. Hence by Theorem \ref{t-4.2}, $ Z(g)\cap Z(h) \neq\emptyset$ and $ Z(g)\cup Z(h)=X $. From this it follows that $ Z(g) $ contains at least two distinct points  and also int$ Z(g)\neq \emptyset $. Consequently $ Z(f) $ also contains at least two distinct points and int$Z(f)\neq\emptyset$. Thus for an arbitrary vertex $ f $ in the comaximal graph $ \Gamma_2^{'}(C(X))$ of $ C(X) $, neither interior of $ Z(f) $ is empty nor $ Z(f) $ is a singleton. We now apply Proposition 2.3 in \cite{ref6}, to conclude that $ e(f)=3 $. Thus radius of $ \Gamma_2^{'}(C(X))=3 $. It follows from Theorem 5.6 in \cite{ref6} that $ X $ is an almost $ P $-space, devoid of any isolated point. Thus the first three statements $ (1),(2),(3) $ are equivalent.

The implication relation $ (4)\Rightarrow(2) $ and $ (3)\Rightarrow(4) $ could be established by following the arguments in Theorem 5.6 in \cite{ref6}.
\end{proof}
\section{Zero-divisor graph versus comaximal graph in $ \ccx $/ $ C(X) $}
	Since for a zero-dimensional space $ X $, the two graphs $ \gnccx $ and $ \gccx $ are syntactically different, it is desirable that for a large class of spaces $ X $, these two graphs are algebraically different. This  means that there should not exist any graph isomorphism from $ \gnccx $ onto $ \gccx $. Any graph isomorphism from a graph $ G_1 $ onto a graph $ G_2 $ stands for a bijective map between the vertices of  $ G_1 $ and $ G_2 $ which further preserves the adjacency relation. This section begins with a natural class of zero-dimensional space $ X $, for which $ \gnccx $ and $ \gccx$ are non isomorphic as graphs.
	 \begin{theorem}\label{to-5.1}
	 	Let $ X $ be a perfectly normal strongly zero-dimensional space such that $ X $ is not a $ P $-space (there are enough examples of such spaces $ X $, viz. if $ X $ is a dense subset of an Euclidean space $ \mathbb{R}^n,\ n\in \mathbb{N} $ for which $ \mathbb{R}^n\setminus X $ is also dense in $ \mathbb{R}^n $, then $ X $ is an example of such a space). Then $ \gccx $ is not a complemented graph, although $\gnccx$ is a complemented graph. 
	 \end{theorem}
	\begin{proof}
		Since $ X $ is not a $ P $-space, it follows from Theorem \ref{t-4.10} that $ \gccx $ is not a complemented graph. Since $ X $ is strongly zero-dimensional, to show that $ \gnccx $ is a complemented graph, it  is equivalent to showing in view of Theorem \ref{t-3.20} that $ \gncx $ is a complemented graph. For that purpose choose any vertex $ f  $ in the zero-divisor graph $ \Gamma(C(X))$ of the ring $ C(X)$. Therefore $ f $ is a nonzero divisor of zero in the ring $ C(X) $ and consequently, int$ Z(f)\neq \emptyset $ and int$ Z(f)\neq X $. Since $ X $ is perfectly normal, every closed subset of $ X $ is a zero-set in it. Hence there exists $ g\in C(X) $ such that $ Z(g)=X\setminus$int$Z(f) $. We observe that $ \emptyset\neq X\setminus Z(f)\subseteq X\setminus$int$Z(f)=Z(g) $. This implies that int$ Z(g)\neq\emptyset $ and hence $ g $ is a nonzero divisor of zero in $ C(X)$. In other words $ g $ is a vertex of $ \gncx $ and $ Z(f)\cup Z(g)=X $ with int$ Z(f)\ \cap $ int$ Z(g)=\emptyset $. This shows that $ g $ is orthogonal to $ f $ in the graph $ \gncx $. Thus $ \gncx $ is a complemented graph.
	\end{proof}
		The next example is of a finite (zero-dimensional) space $ X $ for which the zero-divisor graph of $ C(X)\ (=\ccx) $ is the same as the comaximal graph of $ C(X)\ (=\ccx) $.
		\begin{example}
		Let $ X=\{p,q\} $, a two member space. Then $ X $ is a discrete space and $ C(X) $ can be identified with the ring $ \R\times \R $ = direct product of the ring $ \R $ with itself. Here $ \{0\}\times \R $ and $ \R\times \{0\} $ are the only nonzero proper ideals in $ \R\times \R $. Hence these are the only maximal ideals in  $ \R\times\R $. Also $ C(X)=\ccx $. It is easy to check that the set $ \{(r,0):r\in \R\setminus\{0\}\}\cup\{(0,r):r\in \R\setminus\{0\}\} $ is identical to the set of all vertices in either of the two graphs $ \Gamma(C(X)) $ and $ \Gamma_2^{'}(C(X)) $. Furthermore, for any two distinct vertices $ f $ and $ g $ in any of these two graphs, $ f $ is adjacent to $ g $ in the zero-divisor graph $ \gnccx $ if and only if there exist $ r\neq 0 $, $ s\neq0 $ in $ \R $ such that $ f=(r,0) $ and $ g=(0,s) $ (or $ f=(0,s) $ and $ g=(r,0) $). This is the case if and only if no maximal ideal in $ \R\times\R $ contains both of $ f $ and $ g $, i.e., if and only if $ f $ is adjacent to $ g $ in the comaximal graph $  \Gamma_2^{'}(C(X)) $. Thus it is proved that the identity map
		
		\begin{alignat*}{2}
				I:&V_0(C(X))\rightarrow  V_2(C(X))\\
			&\ \ \ \ \ \ \ \ \ \ \ \ 	f\mapsto f
		\end{alignat*}
	is a graph isomorphism on $ \Gamma(C(X)) $ onto $ \Gamma_2^{'}(C(X)) $. Here $ V_0(C(X)) $ stands for the set of all vertices in $ \Gamma(C(X)) $ with an analogous meaning for $ V_2(C(X)) $.
		\end{example}
		If $ X $ contains finitely many points only, say $ |X|=n $ with $n \geq 3$, $ n\in \mathbb{N} $ then $ X $ is still a discrete space and $ C(X) $ can be identified with the ring $ \R\times\R\times...\times\R $ (direct product of $ \R $ with itself `n' times). It is easy to check that the set of all vertices in the zero-divisor graph $ \Gamma(C(X)) $ is identical to the set of all vertices in $ \gcx $. We may be tempted to believe that the identity map 	\begin{alignat*}{2}
			I:&\vn\rightarrow \vtccx \\
		&\ \ \ \ \ \ \ \ \ \ \ \ 	f\mapsto f
		\end{alignat*} is, this time also, a graph isomorphism. We now show that this is not the case. We shall demonstrate a negative answer to the above belief with the case $ n=3 $. The fact that identity is no longer a graph isomorphism on $ \Gamma(C(X)) $ onto $ \Gamma_2^{'}(C(X)) $, whenever $ |X|=n $ with $ n\in\mathbb{N} $ chosen arbitrary, $ n\geq3 $, can be established by analogous reasoning. For that purpose we simply observe that $ (0,1,0) $ and $ (1,0,0) $ are two different vertices in $ \R\times\R\times \R $ (equipped with both zero-divisor graph and comaximal graph respectively). These two vertices are surely adjacent in the zero-divisor graph of $ C(X) $ (with $ |X|=3 $). Since $ \R\times\R\times\{0\} $, is a maximal ideal in $ \R\times\R\times \R $ containing each of $ (0,1,0) $ and $ (1,0,0) $, it follows that these two vertices are not adjacent in the comaximal graph of $ C(X) $. Thus the identity map 
		\begin{alignat*}{2}
		I:&V_0(C(X))\rightarrow V_2(C(X)) \\
		&\ \ \ \ \ \ \ \ \ \ \ \ 	f\mapsto f
	\end{alignat*}
fails to preserve the adjacency relation and is hence not an isomorphism.	
	
	We are now going to show that if $ X $ is a discrete topological space containing atmost countably infinitely many points, then $ \Gamma(C(X))=\gnccx $ and $ \Gamma_2^{'}(C(X))=\gccx $ are indeed isomorphic as graphs, though of course, such an isomorphism is different from the identity isomorphism if $ |X|\geq3 $, as observed above. For that purpose we need to introduce a kind of quotient graph of an arbitrary graph $ G $ and then establish a result which tells under certain conditions a graph isomorphism between the quotient graphs $ \w{G} $ and $ \w{H} $ will imply a graph isomorphism between two graphs $ G $ and $ H $. This will ultimately lead to a graph isomorphism from $ \Gamma(C(X)) $ onto $ \Gamma_2^{'}(C(X)) $ for a discret space $ X $ with $ |X|\leq \aleph_0 $. Let us now take up the formal construction of the quotient $ \w{G} $ of a simple graph $ G $. Let $ V(G) $ be the set of all vertices in $ G $. For $ x\in V(G) $, set $ [x]=\{y\in V(G): x\sim y\} $. Define a binary relation ``$\thickapprox $'' on $ V(G) $ as follows: for $ x,y \in V(G) $, $ x\thickapprox y $ if and only if $ [x]=[y] $. Then ``$\thickapprox  $" is an equivalence relation on $ V(G) $. For $ x\in V(G) $, we write $ \w{x}=\{y\in V(G): x\thickapprox y\} $. Thus the quotient set $ V(G)/_\thickapprox =\{\w{x}:x\in V(G)\}$. For notational simplicity we write $ V(G)/_\thickapprox=\w{G} $. We now initiate a new graph, whose set of vertices is $ \w{G} $ and the adjacency relation is defined as follows: for $ \w{x} ,\ \w{y}\in \w{G}$ with $ \w{x}\neq\w{y} $, we write $ \w{x} \sim \w{y}$ (and call $ \w{x} $ and $ \w{y} $ adjacent) if and only if $ x\sim y $ in the original graph $ G $. The following result demonstrates that the definition of `$ \sim $' in $ \w{G} $ does not suffer from any ambiguity.
	\begin{theorem}\label{to-5.3}
		Let $ \w{x}\sim \w{y} $ in $ \w{G} $ and $ a\in \w{x},\ b\in \w{y} $. Then $ \w{a}\sim \w{b} $.
	\end{theorem}
\begin{proof}
	$ a\in\w{x}\Rightarrow [a]=[x] $. On the other hand, $ \w{x}\sim \w{y}\Rightarrow x\sim y $. It follows therefore that $ a\sim y $. Again $ b\in \w{y}\Rightarrow [b]=[y] $. Hence $ a\sim b $ and consequently $ \w{a}\sim \w{b} $.
\end{proof}
We designate $ \w{G} $ as the quotient of the graph $ G $.
\begin{theorem}\label{to-5.4}
	Let $ G_1 $ and $ G_2 $ be two simple graphs which are isomorphic under a graph isomorphism $ \psi: V(G_1)\rightarrow V(G_2)$. Then the map $ \Phi: \w{G}_1 \rightarrow \w{G}_2$ defined by $ \Phi(\w{x})=\w{\psi(x)} $ is a  graph isomorphism onto $ \w{G}_2 $. Furthermore, for any $ x\in G_1 $, $ \psi([x])=[\psi(x)] $ and $ |\w{x}|=|\w{\psi(x)}| $, $ |Y| $ designating the cardinal number of the set $ Y $.
\end{theorem}
\begin{proof}
Let $ \w{x},\ \w{y}\in \w{G}_1 $ be such that $ \w{\psi(x)}=\w{\psi(y)} $, then $ [\psi(x)]=[\psi(y)] $. The last relation holds if and only if $ [x]=[y] $ because of the fact that $ \psi $ is a graph isomorphism from $ V(G_1) $ onto $ V(G_2) $. This is true when and only when $ \w{x}=\w{y} $. This settles that the map $ \Phi: \w{G}_1\rightarrow \w{G}_2 $ defined above is a bijection between the vertices of these two quotient graphs. Furthermore, $ \w{x}\sim\w{y} $ in $ \w{G}_1 $ ($ x,y\in G_1 $) $ \Leftrightarrow x \sim y$ in $ G_1 \Leftrightarrow \psi(x)\sim \psi(y)$ in $ G_2 $ (as $ \psi: V(G_1)\rightarrow V(G_2) $ is a bijection) $ \Leftrightarrow \w{\psi(x)}\sim \w{\psi(y)} $ in $ \w{G}_2 \Leftrightarrow \Phi(\w{x})\sim \phi(\w{y})$ in $ G_2 $. Thus it is proved that $ \Phi: \w{G}_1\rightarrow \w{G}_2 $ is a graph isomorphism. The remaining parts of the theorem are straightforward consequence of the fact that $ \psi: V(G_1)\rightarrow V(G_2) $ is a graph isomorphism.
\end{proof}
The following Banach-Stone like theorem ascertaining that under certain conditions, for a pair of graphs $ G_1 $ and $ G_2 $, the existence of a graph isomorphism: $ \w{G}_1 \rightarrow \w{G}_2$ implies that of a graph isomorphism: $ G_1\rightarrow G_2 $.
\begin{theorem}\label{to-5.5}
	 Suppose $ G_1 $ and $ G_2 $ are two simple graphs with the following properties:
	 \begin{enumerate}
	 	\item  There exists a graph isomorphism $ \Phi:\w{G}_1\rightarrow \w{G}_2 $ and \item For each $ \w{x}\in \w{G}_1 $ $ ( $where $ x\in G_1$$ ) $, $ |\Phi(\w{x})|=|\w{x}| $.  Then there can be defined a graph isomorphism from $ G_1 $ onto $ G_2 $.
	 \end{enumerate}
\end{theorem}
\begin{proof}
	Let $ \w{x}\in \w{G}_1 $, then $ \w{x} $ is actually an equivalence class of elements in $ G_1 $. We choose exactly one element $ x_\lambda $ from this equivalence class and this we do for each $ \w{x}\in \w{G}_1 $. This defines a set $ \{x_\lambda:\lambda\in \Lambda\} $ of members of $ G_1 $. Thus for $ \lambda,\mu\in \Lambda $, $ \w{x}_\lambda=\w{x}_\mu $ if and  only if $ \lambda=\mu $. By our hypothesis (2), there exists for each $ \lambda\in \Lambda $, a bijection $ \Phi_\lambda:\w{x}_\lambda\rightarrow \Phi(\w{x}_\lambda) $. It is also clear that $ \w{G}_2=\{\Phi(\w{x}_\lambda):\lambda\in \Lambda\} $. Define a map $ \psi: V(G_1)\rightarrow V(G_2) $  as follows: if $ x\in V(G_1) $, then there exists a unique $ \lambda\in \Lambda $ such that $ x\in \w{x}_\lambda $. We set $ \psi(x) =\Phi_\lambda(x)$. Then $ \psi $  is a well defined map from $ V(G_1) $ onto $ V(G_2) $.
	
	To show that $ \psi $ is a one-to-one map, choose any two distinct vertices $ a $ and $ b $ from $ V(G_1) $. If $ a $ and $ b $ are contained in the same equivalence class $ \w{x}_\lambda $, $ \lambda\in \Lambda $, then as $ \Phi_\lambda $ is one-to-one it follows that $ \Phi_\lambda(a)\neq \Phi_\lambda(b) $ and this implies that $ \psi(a)\neq\psi(b) $. Assume therefore that $ a $ and $ b $ belong to different equivelence classes: $a\in \w{x}_\lambda $ and $b\in \w{x}_\mu $ with $ \lambda\neq\mu $ in $ \Lambda $. As $ \Phi:\w{G}_1\rightarrow \w{G}_2  $ is a  one-to-one map, this implies that $ \Phi(\w{x}_\lambda)\neq \Phi(\w{x}_\mu) $, furthermore, $ \psi(a)=\Phi_\lambda(a)\in \Phi(\w{x}_\lambda) $. Analogously $ \psi(b)\in \Phi(\w{x}_\mu) $. Since the equivalence classes $ \Phi(\w{x}_\lambda) $ and $ \Phi(\w{x}_\mu) $ in $ G_2 $ are disjoint, it follows that $ \psi(a)\neq \psi(b) $. Thus $ \psi $ is a bijection on $ V(G_1) $ onto $ V(G_2) $. To complete the proof it remains to show that $\psi$ preserves the adjacency relation. So let $ a $ and $ b $ be adjacent vertices in $ G_1 $, i.e., $ a\sim b $. Then there exist $ \lambda,\ \mu\in \Lambda $ such that $ a\in \w{x}_\lambda $ and $ b\in \w{x}_\mu $. Since no two elements in the same equivalence class $ \w{x} $ in $ V(G_1) $ can be adjacent in the original graph $ G_1 $, a fact which can be easily checked because no vertices in $ G_1 $ can be adjacent to itself, it follows that $ \lambda\neq\mu $. This yields that $ \w{x}_\lambda\sim \w{x}_\mu $ in $ \w{G}_1 $. Since $ \Phi $ is graph isomorphism, this implies that $ \Phi(\w{x}_\lambda)\sim \Phi(\w{x}_\mu) $ in $ \w{G}_2 $. But as $ \psi(a)=\Phi_\lambda(a)\in \Phi(\w{x}_\lambda) $ and analogously $ \psi(b)\in \Phi(\w{x}_\mu) $ it follows that $ \psi(a)\sim\psi(b) $ in $ V(G_2) $. Conversely, if $ \psi(a) \sim \psi(b) $ in $ V(G_2) $ for $ a,\ b\in V(G_1) $, then it can be proved by reversing back the above chain of arguments that $ a\sim b $ in $ V(G_1) $. Thus $ \psi: V(G_1)\rightarrow V(G_2) $ is a graph isomorphism from the graph $ G_1 $ onto the graph $ G_2 $.
\end{proof}
\begin{theorem}\label{to-5.6}
	Let $ X $ be a discrete topological space. Let $ G_1 $ and $ G_2 $ designate the zero-divisor graph of $ C(X) $ and the comaximal graph of $ C(X) $ respectively. Then the quotient graphs $ \w{G}_1 $ and $ \w{G}_2 $ are graph isomorphic.
\end{theorem}
\begin{proof}
	It is easy to check that $ V(G_1) =$ the set of all vertices in $ G_1 =V(G_2)=$ the set of all vertices in $ G_2=\{f\in C(X):\emptyset\neq Z(f)\ \text{and}\ Z(f)\neq X\} $. It is clear that for any $ f\in V(G_1)=V(G_2) $, $ \w{f}=\w{\chi}_{X\setminus Z(f)} =$ the characteristic function of $ X\setminus Z(f) $. Therefore $ \w{G}_1=\w{G}_2 =\{\w{\chi}_{X\setminus Z(f)}: f$ is a nonzero non-unit in $C(X)\}$. Let $ \Phi:\w{G}_1\rightarrow \w{G}_2 $ be the map defined as follows: $ \Phi(\w{\chi}_{X\setminus Z(f)})=\w{\chi_{Z(f)}} $. Then clearly $ \Phi $ is a bijection between the vertices $ V(\w{G}_1) $ and $ V(\w{G}_2) $ of the quotient graphs $ \w{G}_1 $ and $ \w{G}_2 $. Furthermore, for nonzero non-units $ f,g $ in $ C(X) $, $ \w{\chi}_{X\setminus Z(f)}\sim \w{\chi}_{X\setminus Z(g)} $ in $ \w{G}_1 $ if and only if $ \w{f}\sim \w{g} $ (recall that $ \w{f}=\w{\chi}_{X\setminus Z(f)} $, a relation obtained earlier) if and only if $ Z(f) \cup Z(g)=X$ if and only if $ (X\setminus Z(f))\cap (X\setminus Z(g))=\emptyset $ if and only if $ Z({\chi}_{ Z(f)})\cap Z({\chi}_{Z(g)})=\emptyset$. The last relation is true if and only if $ {\chi}_{ Z(f)}\sim {\chi}_{Z(g)} $ in the comaximal graph $ G_2 $. Surely this is the case when and only when $ {\w{\chi}_{Z(f)}}\sim {\w{\chi}_{Z(g)}} $ in $ \w{G}_2 $, i.e., if and only if $\Phi(\w{\chi}_{X\setminus Z(f)})\sim \Phi(\w{\chi}_{X\setminus Z(g)})$ in $ \w{G}_2 $. Thus it is proved that both of $\Phi:\w{G}_1\rightarrow \w{G}_2$ and $ \Phi^{-1}: \w{G}_2\rightarrow \w{G}_1 $ preserve the adjacency relation. Hence $ \Phi: \w{G}_1\rightarrow \w{G}_2 $ defined above is a graph isomorphism.
\end{proof}
\begin{remark}\label{ro-5.7}
	Since for a $ P $-space $ X $, the nonzero non-units in $ C(X) $ are the same as the nonzero divisors of zero in the ring and for $ f\in C(X) $, $ Z(f) $ is a clopen set in $ X $, a careful look into the above proof yields the following more general fact: If $ X $ is a $ P $-space and $ G_1 $ and $ G_2 $ are the zero-divisor graph and comaximal graph of $ C(X) $ respectively, then the quotient graphs $ \w{G}_1 $ and $ \w{G}_2 $ are isomorphic.
\end{remark}
\begin{remark}\label{to-5.8}
 In the above Theorem \ref{to-5.6}, for a discrete topological space $ X $, in the quotient graph $ \w{G}_1 $, for any $ f\in C(X) $, $ \w{f} =\{g\in C(X): Z(g)=Z(f)\}$. Therefore the cardinal number of the set $ \w{f}=|\w{f}|=|\R^{X\setminus Z(f)}|=c^{|X\setminus Z(f)|} $. It is easy to check that $ |\w{\chi}_{ Z(f)}| $ in the quotient graph $ \w{G}_2=|\R^{Z(f)}|=c^{|Z(f)|} $. If we combine Theorem \ref{to-5.5} and Theorem \ref{to-5.6} and take into consideration for any $ f\in C(X) $, the cardinal number $ |\w{f}| $ and $ |\w{\chi}_{Z(f)}| $ in the quotient graphs $ \w{G}_1 $ and $ \w{G}_2 $ respectively, then this yields the following result:
\end{remark}
\begin{theorem}\label{to-5.9}
	Suppose $ G_1 $ and $ G_2 $ are respectively the zero-divisor graph and the comaximal graph of $ C(X) $ with $ X $ a discrete topological space. Suppose further that for any $ f\in C(X) $ with $ f\neq0 $ and $ Z(f)\neq\emptyset $, $ c^{|Z(f)|}=c^{|X\setminus Z(f)|} $, here $ c $ is the cardinal number of the continuum. Then $ G_1 $ and $ G_2 $ are isomorphic as graphs.
\end{theorem}
Now if the discrete space $ X $ is atmost countable then for any $ f\in C(X) $, $ f\neq0 $, $ Z(f)\neq\emptyset $, $ |Z(f)|\leq\aleph_0 $ and $ |X\setminus Z(f)|\leq\aleph_0 $. Consequently, $ c^{|Z(f)|}=c^{|X\setminus Z(f)|}=c $. Hence the following proposition results:
\begin{theorem}\label{to-5.10}
	If the discrete space $ X $ is atmost countable $ ( $i.e., $ X $ is either finite or a countably infinite space$ ) ,$ then the zero-divisor graph $ G_1 $ and the comaximal graph $ G_2 $ of the ring $ C(X) $ are isomorphic.
\end{theorem}
An explicit representation for an isomorphism between $ G_1=\gncx=\gnccx $ and $ G_2=\gcx=\gccx $ with $ |X|=3 $ is provided by the following diagram.

\begin{example}\label{e-0.5}
	Suppose $ |X|=3 $, then $ G_1=\gncx=\gnccx $ and $G_2= \gcx=\gccx $ are isomorphic.
\end{example}
\begin{proof}
	$ C(X) = \mathbb{R}\times\mathbb{R}\times\mathbb{R}$. $ V(G_1)=V(G_2)=\{(x,y,z)\in\mathbb{R}^3: \text{x,\ y,\ z are not all simultaneously equal to zero and not all nonzero}\}=V $, say. Then	$ \widehat{G}_1= \{\widehat{\xi}: \xi=(r_1,r_2,r_3)\in V: r_1,r_2,r_3\in\{0,1\}\}=\widehat{G}_2 $. Here $ |\widehat{G}_1|=|\widehat{G}_2|=6 $. Define $ \Phi: \widehat{G}_1\rightarrow \widehat{G}_2 $ as follows $ \Phi(\widehat{(r_1,r_2,r_3)})=\widehat{(s_1,s_2,s_3)} $, where $ s_i=0$ if $r_i=1 $ and $ s_i=1$ if $r_i=0 $. This is the desired graph isomorphism.

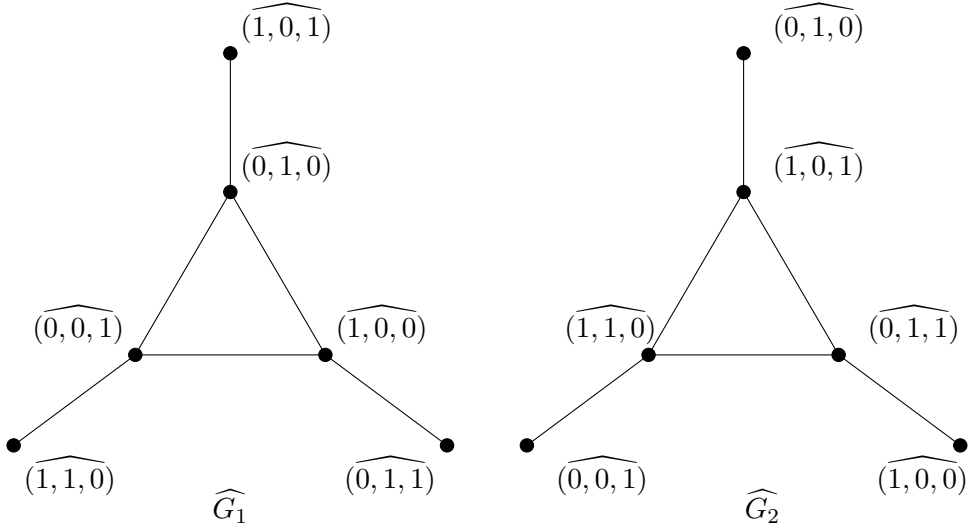
\begin{figure*}[!]
	\begin{tikzpicture}
		\draw (-2.,0.)-- (0.5,0);
		\draw(-2,0)--(-.75,2.16);
		\draw(-.75,2.16)--(.5,0);
		\draw (-.75,2.16)--(-.75,4.);
		
		\draw(-2.,0)--(-3.6,-1.2);
		\draw (.5,0)--(2.1,-1.2);
		
		\draw (-.75,4) node[anchor=south west] {$\widehat{(1,0,1)}$};
		\draw (-.75,2.16) node[anchor=south west] {$\w{(0,1,0)}$};
		\draw (-2,0) node[anchor=south east] {$\w{(0,0,1)}$};
		
		\draw (-3.6,-1.2) node[anchor=north west] {$\w{(1,1,0)}$};
		\draw (2.1,-1.2) node[anchor=north east] {$\w{(0,1,1)}$};
		\draw (.5,0) node[anchor=south west] {$\w{(1,0,0)}$};
		
		\draw (-.75,-2) node[] {$\w{G_1}$};
		
		\draw (4.75,0.)-- (7.25,0);
		\draw(4.75,0)--(6.,2.16);
		\draw(3.15,-1.2)--(4.75,0);
		\draw(6.,2.16)--(7.25,0);
		\draw (6.,2.16)--(6.,4.);
		\draw (7.25,0)--(8.85,-1.2);
		
		\draw (6.25,4) node[anchor=south west] {$\widehat{(0,1,0)}$};
		\draw (6.25,2.16) node[anchor=south west] {$\w{(1,0,1)}$};
		\draw (5,0) node[anchor=south east] {$\w{(1,1,0)}$};
		
		\draw (3.4,-1.2) node[anchor=north west] {$\w{(0,0,1)}$};
		\draw (9.1,-1.2) node[anchor=north east] {$\w{(1,0,0)}$};
		\draw (7.5,0) node[anchor=south west] {$\w{(0,1,1)}$};
		
		\draw (6.25,-2) node[] {$\w{G_2}$};
		
		\begin{scriptsize}
			\draw [fill=black] (-2,0) circle (2.5pt);
			\draw [fill=black] (.5,0) circle (2.5pt);
			\draw [fill=black] (-.75,2.16) circle (2.5pt);
			\draw [fill=black] (-.75, 4) circle (2.5pt);
			\draw [fill=black] (-3.6,-1.2) circle (2.5pt);
			\draw [fill=black] (2.1,-1.2) circle (2.5pt);
			
			\draw [fill=black] (4.75,0.) circle (2.5pt);
			\draw [fill=black] (7.25,0) circle (2.5pt);
			\draw [fill=black] (6.,2.16) circle (2.5pt);
			\draw [fill=black] (3.15,-1.2) circle (2.5pt);
			\draw [fill=black] (8.85,-1.2) circle (2.5pt);
			\draw [fill=black] (6.,4.) circle (2.5pt);
		\end{scriptsize}
	\end{tikzpicture}
\end{figure*}
\end{proof}

\newpage
The next theorem shows that the countable hypothesis of the set in Theorem \ref{to-5.10} can not be dropped.
\begin{theorem}\label{to-5.11}
	We assume continumm hypothesis $($CH $)$. Suppose $ X $ is an uncountable discrete topological space. Then the zero-divisor graph $ G_1=\Gamma(C(X)) $ and the comaximal graph $ G_2=\Gamma_2^{'}(C(X)) $ of $ C(X) $ are not isomorphic.
\end{theorem}
\begin{proof}
	If possible let there exist a graph isomorphism $ \psi:G_1\rightarrow G_2 $. Choose an $ f\in C(X) $ with $ Z(f)=\{p\} $ for some point $ p $ in $ X $. Then by Theorem \ref{to-5.4}, $$ |\w{f}|=|\w{\psi(f)}|\  \text{and}\  |[f]|=|[\psi(f)]|, $$ the notations having their usual meaning as explained while defining the quotient graph $ \w{G} $ of a simple graph $ G $. The first relation implies in view of Remark \ref{to-5.8} that $ c^{X\setminus Z(f)}=c^{X\setminus Z(\psi(f))} $. Since $$ |X\setminus Z(f)|=|X\setminus \{p\}|=|X|\geq c\  \text{(by CH)},$$ it follows therefore that $ |X\setminus Z(\psi(f))|\geq c $ (the reason is that $c^{\aleph_0}=c$ and $ c^c=2^c>c $). To get the desired contradiction we shall show that $ |[f]|\neq|[\psi(f)]| $. Indeed we shall show that $ |[f]|<|[\psi(f)]| $. For that purpose we observe that $ [f]=\{g\in C(X):gf=0\}=\{g\in C(X): g(x)f(x)=0$ for each $x\in Z(f)$ and $g(x)=0$ for each $x\in X\setminus Z(f)\}=\{g\in C(X): g(x)=0$ for each $x\in X\setminus\{p\}\} $ and therefore $ |[f]|=c $. On the other hand, in the comaximal graph of $C(X) $ for any $ h\in C(X) $, $ h\in [\psi(f)] $ if and only if $ Z(h)\cap Z(\psi(f))=\emptyset $. Therefore $ [\psi(f)]=\{h\in C(X): h$ takes nonzero values at each point in $Z(\psi(f)) \} \geq c^{|X\setminus Z(\psi(f))|}\geq c^c$, as $ |X\setminus Z(\psi(f))|\geq c $, observed earlier. Thus we get $ |[\psi(f)]|\geq c^c>c =|[f]|$.
\end{proof}
Theorem \ref{to-5.10} and Theorem \ref{to-5.11} combined together yield the following proposition.
\begin{theorem}\label{to-5.12}
$ ( $CH $ ) $	A discrete topological space $ X $ is atmost countable if and only if the zero-divisor graph and the comaximal graph of $ C(X) $ are isomorphic.
\end{theorem}
	A careful scrutiny into the proof of Theorem \ref{to-5.11} leads to the following result.
\begin{theorem}\label{to-5.14}
	$ ( $CH $ ) $ For an uncountable discrete topological space $ X $, the zero-divisor graph $ \gnccx $ is not isomorphic to the comaximal graph $ \gccx $ of the ring $ \ccx $.
\end{theorem}
We conclude this section by combining all the above mentioned theorems.
\begin{theorem}\label{to-5.15}
	$ ( $CH $ ) $ For a discrete topological space $ X $, the following statements are equivalent:\begin{enumerate}
		\item The zero-divisor graph of $ C(X) $ is isomorphic to the comaximal graph of $ C(X) $.
		\item The zero-divisor graph of $ \ccx$ is isomorphic to the comaximal graph of $ \ccx $.
		\item $ X $ is atmost a countable set.
	\end{enumerate}
\end{theorem}

\end{document}